\documentclass[12pt,a4paper]{article}

\usepackage{graphicx}
\usepackage[leqno]{amsmath}
\usepackage{amssymb,amsthm,upref,amscd}
\usepackage[T1]{fontenc}
\usepackage{times}
\usepackage{amsfonts}
\newtheorem{Thm}{Theorem}[section]

\newtheorem{Lem}[Thm]{Lemma}

\newtheorem{Prop}[Thm]{Proposition}
\newtheorem{Rem}[Thm]{Remark}

\numberwithin{equation}{section}

\newcommand{\R}{\mathbb{R}}

\newcommand{\Z}{\mathbb{Z}}

\newcommand{\cC}{{\mathcal C}}

\newcommand{\cS}{{\mathcal S}}

\newcommand{\eps}{\varepsilon}
\newcommand{\al}{\alpha}
\newcommand{\be}{\beta}

\newcommand{\de}{\delta}
\newcommand{\De}{\Delta}
\newcommand{\la}{\lambda}
\newcommand{\La}{\Lambda}

\newcommand{\Om}{\Omega}
\newcommand{\om}{\omega}

\newcommand{\opint}{\text{\rm int\,}}

\newcommand{\id}{\text{\rm id}}
\newcommand{\pa}{\partial}
\newcommand{\ind}{\text{\rm ind}}
\newcommand{\essinf}{\text{\rm ess\,inf}}

\newcommand\loc{\mathrm{loc}}
\newcommand\meas{\mathrm{meas}}

\newcommand{\abs}[1]{\lvert#1\rvert}

\newcommand{\weakto}{\rightharpoonup}

\newenvironment{altproof}[1]
{\noindent
{\em Proof of {#1}}.}
{\nopagebreak\mbox{}\hfill $\Box$\par\addvspace{0.5cm}}

\begin{document}

\title{Nonlinear Schr\"odinger equations near an infinite well potential}

\author{\sc{Thomas Bartsch} \and \sc{Mona Parnet}}

\date{}
\maketitle

\begin{abstract}
The paper deals with standing wave solutions of the 
dimensionless nonlinear Schr\"{o}dinger equation
\begin{equation}\label{eq:abs1}
i\Phi_t(x,t) = -\Delta_x\Phi +V_\la(x)\Phi + f(x,\Phi),
 \quad x\in\R^N,\ t\in\R,\tag{$NLS_\la$}
\end{equation}
where the potential $V_\la:\R^N\to\R$ is close to an infinite well 
potential $V_\infty:\R^N\to\R$, i.~e.\ $V_\infty=\infty$ on an 
exterior domain $\R^N\setminus\Om$, $V_\infty|_\Om\in L^\infty(\Om)$,
and $V_\la\to V_\infty$ as $\la\to\infty$ in a sense to be made precise.
The nonlinearity may be of Gross-Pitaevskii type. A solution of 
\eqref{eq:abs1} with $\la=\infty$ vanishes on $\R^N\setminus\Om$ and 
satisfies Dirichlet boundary conditions, hence it solves
\begin{equation}\label{eq:abs2}
\left\{
\begin{aligned}
i\Phi_t(x,t) &= -\Delta_x\Phi +V_\la(x)\Phi + f(x,\Phi),
        &&\quad x\in\Om,\ t\in\R\\
\Phi(x,t) &= 0 &&\quad x\in\pa\Om,\ t\in\R.
\end{aligned}
\tag{$NLS_\infty$}
\right.
\end{equation}
We investigate when a solution $\Phi_\infty$ of the infinite well
potential \eqref{eq:abs2} gives rise to nearby solutions $\Phi_\la$ of 
the finite well potential \eqref{eq:abs1} with $\la\gg1$ large. 
Considering \eqref{eq:abs2} as a singular limit of \eqref{eq:abs1} we
prove a kind of singular continuation type results.
\end{abstract}


{\bf Keywords}: nonlinear Schr\"odinger equations, infinite well potential,
deep potential well, nonlinear eigenvalue problems, singular limit, 
variational methods, topological methods, singular continuation\\

{\bf  AMS subject classification}: 35J20, 35J61, 35J91, 35Q55, 58E05


%
\section{Introduction}\label{sec:intro}
Infinite well potentials like the infinite square well or the
infinite spherical well are helpful as instructive models 
to describe confined particles in quantum mechanical systems.
They are often used as a starting point for solving finite well 
problems. In this paper we investigate nonlinear Schr\"odinger
equations, like the Gross-Pitaevskii equation, with a potential 
$V_\la:\R^N\to\R$ close to an infinite well potential 
$V_\infty:\R^N\to\R$. More precisely, $V_\infty=\infty$ on an 
exterior domain $\R^N\setminus\Om$, and $V_\infty|_\Om\in L^\infty(\Om)$. 
As $\la\to\infty$ the potential depth of $V_\la$ becomes infinite, 
i.~e.\ $V_\la\to V_\infty$, in a sense to be made precise below. Our
goal is to give rigorous proofs for the passage from the infinite 
well potential to the finite well potential. \\

We are interested in standing waves $\Phi(t,x)=e^{i\om t}u(x)$ of the 
finite well nonlinear Schr\"odinger equation
\begin{equation}\label{eq:NLS-finite}
i\Phi_t(x,t) = -\Delta_x\Phi(x,t)+V_\la(x)\Phi+f(x,\Phi),
\quad x\in\R^N,\ t\in\R,
\tag{$NLS_\la$}
\end{equation}
where $V_\la(x)\to V_\infty(x)$ as $\la\to\infty$. For $\la=\infty$ 
a solution should vanish in $\R^N\setminus\Om$ and satisfy Dirichlet
boundary conditions on $\Om$, hence it is a solution of the singular limit 
problem:
\begin{equation}\label{eq:NLS-infinite}
\left\{
\begin{aligned}
i\Phi_t(x,t) & = -\Delta_x\Phi +V_\la(x)\Phi + f(x,\Phi),
             &&\quad x\in\Om,\ t\in\R,\\
\Phi(x,t) &= 0 &&\quad x\in\pa\Om,\ t\in\R.
\end{aligned}
\right.
\tag{$NLS_\infty$}
\end{equation}
The question we address in this paper is: suppose we know a solution 
$\Phi_\infty$ of \eqref{eq:NLS-infinite}, does there exist a nearby 
solution $\Phi_\la$ of \eqref{eq:NLS-finite} for $\la$ large?\\

Standing wave solutions of \eqref{eq:NLS-finite} correspond to 
solutions of the stationary nonlinear Schr\"{o}dinger equation
\begin{equation}\label{eq:NLS1}
\left\{
\begin{aligned}
& -\De u+V_\la(x)u = f(x,u)&&\quad\text{for } x\in\R^N;\\
& u(x)\to 0 &&\quad\text{as } |x|\to \infty,
\end{aligned} \tag{$S_\la$}
\right.
\end{equation}
where we incorporated the term $\om u$ generated by the ansatz into 
the potential without changing notation. For $\la=\infty$ we are 
similarly led to consider
\begin{equation}\label{eq:limit1}
-\De u+V_\infty(x)u = f(x,u),\quad u\in H^1_0(\Omega),\tag{$S_\infty$}
\end{equation}
as a singular limit of $(S_\la)$ as $\la \to \infty$. The original 
question can now be reformulated as which solutions $u_\infty$ of 
\eqref{eq:limit1} appear as limits of solutions $u_\la$ of 
\eqref{eq:NLS1}. Solutions of \eqref{eq:limit1} can be obtained via 
variational or topological methods. We provide conditions on the
convergence of $V_\la\to V_\infty$ and on $f$ such that an isolated 
solution $u_\infty$ of \eqref{eq:limit1} which can be found by 
variational or topological methods gives rise to a family of 
solutions $u_\la$ of \eqref{eq:NLS1}. We include of course the 
generic case where $u_\infty$ is a nondegenerate solution of 
\eqref{eq:limit1}. \\

For the proofs we develop an abstract functional analytic approach 
in order to deal with the above type of singular limit problem. 
Our results may be thought of as being continuation results near a 
singular limit: For $\la<\infty$ we look for solutions of an 
equation $F_\la(u)=0$ defined on $H^1(\R^N)$, whereas the limit 
equation $F_\infty(u)=0$ is only defined for $u\in H^1_0(\Omega)$.
Some of the methods we develop can also be applied to more general 
nonlinear eigenvalue problems that are not necessarily of 
variational type. \\

The paper is organized as follows. In Section~\ref{sec:NLS} we 
state our main results about \eqref{eq:NLS1}, and we discuss related
results. Then in Section~\ref{sec:crit-point} we formulate the 
functional analytic setting which will be considered throughout the 
paper. Here we also state our main abstract results about solutions
of nonlinear equations near a singular parameter limit. The abstract 
results as well as the results about $(S_\la)$ will be proved in 
sections~\ref{sec:variational}~--~\ref{sec:nondeg}.

\section{NLS near an infinite well potential}\label{sec:NLS}
We begin with collecting our assumptions on the potentials $V_\la$.
These are given in the form  $V_\la= a_0+\la a$, so the problem we
consider is
\begin{equation}\label{eq:NLS}
\left\{
\begin{aligned}
& -\De u+(a_0(x)+\la a(x))u = f(x,u)&&\quad\text{for } x\in\R^N;\\
& u(x)\to 0 &&\quad\text{as } |x|\to \infty,
\end{aligned} 
\tag{$S_\la$}
\right.
\end{equation}
and the limit problem is
\begin{equation}\label{eq:limit}
-\De u+a_0(x)u = f(x,u),\quad u\in H^1_0(\Omega).
\tag{$S_\infty$}
\end{equation}

The distinguishing feature is that the potential 
$a\in L^\infty_{loc}(\R^N)$ satisfies $a\ge 0$ and 
$a^{-1}(0)=\overline\Om$ with $\Om\subset\R^N$ nonempty, open, and 
bounded. Consequently, $V_\la(x)\to\infty$ as $\la\to\infty$ for 
$x\notin\overline\Om$. \\

In order to describe the assumptions on $a$ and $a_0$ we need some 
notation. For $x\in\R^N$ and $r>0$ we set 
$B_r(x):=\{y\in\R^N:|y-x| < r\}$. We also set
$K_r^c:=\{x\in\R^n:|x|_\infty>r\}$. Let
$\mu_N(-\De+V_{\la},G)$ be the infimum of the spectrum of
$-\De+V_{\la}$ on an open subset $G\subset\R^N$ with Neumann boundary
conditions, i.~e.
\[
\mu_N(-\De+V_{\la},G)=\inf_{\psi\in H^1(G)\setminus \{0\}}
\frac{\int_G(|\nabla\psi|^2+V_{\la}\psi^2)dx}{\|\psi\|^2_{L^2(G)}}.
\]
Our basic hypotheses on the potential are:
\begin{itemize}
\item[$(V_1)$] $a_0\in L^{\infty}_\loc(\R^N)$ and
  $\essinf\ a_0 > -\infty$.
\item[$(V_2)$] $a\in L^\infty_\loc(\R^N)$, $a(x)\geq 0$ and
  $\Omega:=\opint\ a^{-1}(0)$ is a non-empty open subset of $\R^N$ with
  Lipschitz boundary.
\item[$(V_3)$] There exists a sequence $R_j\to\infty$ such that
\[
  \lim_{\la\to\infty}\liminf_{j\to\infty}
  \mu_N(-\De+V_{\la},K_{R_j}^c) = \infty.
\]
\end{itemize}

The reader can find a discussion of condition $(V_3)$, in particular various
equivalent conditions, in \cite{bartsch-pankov-wang:2001}. Condition $(V_3)$
holds, for instance, if $a$ satisfies:
\begin{itemize}
\item[$(V_4)$] There exist $M>0$ and $r>0$ such that
\[
 \meas(\{x\in B_r(y): a(x)<M\}) \to 0 \quad\text{as }|y|\to\infty
\]
where $\meas$ denotes the Lebesgue measure.
\end{itemize}
$(V_3)$ implies that the embedding $H^1_0(\Om)\hookrightarrow L^p(\Om)$ is
compact for $2\le p<\infty$. Observe that $(V_3)$ and $(V_4)$ allow that
$\Om$ may be unbounded. For some results we require the stronger 
condition
\begin{itemize}
\item[$(V_5)$] The form domain 
\[
E:=\left\{u\in H^1(\R^N): \int_{\R^N} a_0u^2 < \infty,\ 
           \int_{\R^N} au^2 < \infty\right\}
\]
embeds compactly into $L^2(\R^N)$.
\end{itemize}

This holds, for instance, if $a_0(x)\to\infty$ or $a(x)\to\infty$ as 
$|x|\to\infty$, a condition usually satisfied for confining potentials.
$(V_5)$ also holds under the weaker condition
\begin{itemize}
\item[$(V_6)$] For any $M>0$ and any $r>0$ there holds:
\[
\meas\{x\in B_r(y): a(x\le M\} \to 0\quad\text{as } |y|\to\infty
\]
\end{itemize}

A proof that $(V_6)$ implies $(V_5)$ can be found in 
\cite{kondratev.shubin:1999}; see also \cite{molchanov:1953}.
\\

Concerning the nonlinearity $f$ we only require that

\begin{itemize}
\item[($f_1$)] $f$ is a Carath\'{e}odory function, and there exists
 constants $C>0$,  $2<q<p<2^*$ such that
 \begin{displaymath}
  \abs{f(x,t)}\leq C(\abs{t}^{p-1}+\abs{t}^{q-1})\quad
   \text{for }t\in\R,\text{ a.~e.}\ x\in\R^N.
 \end{displaymath}
\end{itemize}
This includes the model nonlinearity $f(x,u)=W(x)\cdot|u|^{p-2}u$ with 
$2 < p < 2N/(N-2)^+$ and $W\in L^\infty(\R^N)$, which appears in the 
Gross-Pitaevskii equation.\\

We define $E_\infty:=H^1_0(\Om)$ provided with the scalar product
\[
\langle u,v\rangle := \int_{\Om} (\nabla u \nabla v + (b+a_0)uv)\,dx
\]
where $b:=1-\essinf\,a_0$. As a consequence of $(V_1)$ and $(V_2)$ this
induces a norm which is equivalent to the standard norm of $H^1_0(\Om)$.
Setting $F(x,u):=\int_0^u f(x,t)\,dt$, it is well known that the functional
$J_\infty: E_\infty \to \R$ defined by
\[
\begin{aligned}
 J_\infty(u)
  &= \frac12\int_{\Om} (|\nabla u|^2+a_0 u^2)\,dx - \int_{\Om} F(x,u)\,dx\\
  &= \frac12\|u\|^2 - \int_{\Om} \left(\frac{b}{2}u^2 + F(x,u)\right)\,dx
\end{aligned}
\]
is of class $\cC^1$, and that critical points of $J_\infty$ are solutions of
\eqref{eq:limit}.\\

Recall that the critical groups of an isolated critical point $u$ of a
functional $J:E\to\R$ are defined as
$C_k(J,u) := H_k(J^c,J^c\setminus\{u\})$ where $c:=J(u)$. Here $H_*$ is
singular homology with coefficients in a commutative ring $R$ with unit;
typically $R=\Z$ or $R$ is a field.\\

Now we can state our first result.

\begin{Thm}\label{thm:crit-group1}
Assume $(V_1)-(V_3)$ and $(f_1)$ hold. Let $u_\infty\in E_\infty$ be an 
isolated solution of \eqref{eq:limit} with nontrivial critical groups
$C_*(J_\infty,u_\infty)$. Then there exists $\La\ge1$ such that for each
$\la\ge\La$ there exists a solution $u_\la\in E$ of \eqref{eq:NLS} with
$u_\la\to u_\infty$ in $E$ as $\la\to\infty$.
\end{Thm}

\begin{Rem}
If $0$ does not belong to the spectrum of $-\Delta+a_0$ in $H^1_0(\Om)$
then this holds true for $-\Delta+a_0+\la a$ for $\la$ large. Then the 
solutions which we obtain in Theorem \ref{thm:crit-group1} and in the 
theorems below decay exponentially; see \cite{pankov:2008}.
\end{Rem}

If $C_*(J_\infty,u_\infty)=0$ then the solution $u_\infty$ cannot be discovered 
using variational methods, and it can disappear under small perturbations.
In our next result we strengthen the hypotheses by assuming that $u_\infty$
has nontrivial index. Consider the functional
\[
 K_\infty: E_\infty\to\R,\quad
 K_\infty(u) = \int_{\Om}\left(\frac{b}{2}u^2 + F(x,u)\right)\,dx,
\]
and define its gradient $k_\infty=\nabla K_\infty:E_\infty\to E_\infty$ with 
respect to the above scalar product on $E_\infty$. Then $k_\infty$ is 
completely continuous because $p<2^*$ in $(f_1)$. The index of $u_\infty$ is 
then defined by
\[
\ind(k_\infty,u_\infty)
 := \deg(\id - k_\infty, B_\de(u_\infty,E_\infty),0).
\]
Here $\deg$ denotes the Leray-Schauder degree, $\de>0$ is small so 
that $u_\infty$ is the only solution of \eqref{eq:limit} in the 
$\de$-ball $B_\de(u_\infty,E_\infty)$ of $u_\infty$ in $E_\infty$.

\begin{Thm}\label{thm:index1}
Assume $(V_1),(V_2),(V_5)$ and $(f_1)$ hold. Let $u_\infty\in E_\infty$ be an 
isolated solution of \eqref{eq:limit} with nontrivial index
$\ind(k_\infty,u_\infty)$. Then there exists a connected set
\[
\cS \subset \{(\la,u)\in\R\times E:\ u\text{ solves }\eqref{eq:NLS}\}
    \subset \R\times E
\]
such that $\cS$ covers a parameter interval $[\La,\infty)$ for some $\La\ge1$.
Morevover, $u_n\to u_\infty$ for any sequence $(\la_n,u_n)\in E$ with
$\la_n\to\infty$.
\end{Thm}

The assumption $\ind(k_\infty,u_\infty)\ne0$ in Theorem~\ref{thm:index1}
is stronger than the assumption $C_*(J_\infty,u_\infty)\ne0$ in
Theorem~\ref{thm:crit-group1} because of the Poincar\'e-Hopf formula:
\begin{equation}\label{eq:Poincare-Hopf}
\ind(k_\infty,u_\infty)
 = \sum_{i=0}^\infty (-1)^i {\rm rank}\, C_i(J_\infty,u_\infty).
\end{equation}

Surprisingly, the strong assumption $(V_5)$ can be replaced by $(V_3)$ if
$f$ satisfies 

\begin{itemize}
\item[($f'_1$)] $f$ is differentiable in $t$, $f$ and $f_t$ are
 Carath\'{e}odory functions and there exist constants $c>0$,
 $2<q<p<2^*=2N/(N-2)^+$ such that
 \[
  \abs{f_t(x,t)}\leq c(\abs{t}^{p-2}+\abs{t}^{q-2})\quad
   \text{for }t\in\R,\text{ a.~e.}\ x\in\R^N;
 \]
\end{itemize}

With this condition the functional $J_\la$ is of class $C^2$.

\begin{Thm}\label{thm:index1a}
Assume $(V_1)-(V_3)$ and $(f'_1)$ hold. Let $u_\infty\in E_\infty$ be an
isolated solution of \eqref{eq:limit} with nontrivial index
$\ind(k_\infty,u_\infty)$. Then the conclusion of 
Theorem~\ref{thm:index1} holds.
\end{Thm}

For our last result about \eqref{eq:NLS} we consider the case of a 
nondegenerate solution $u_\infty$. 

\begin{Thm}\label{thm:nondeg1}
Assume $(V_1)-(V_3)$ and $(f'_1)$ hold. Let $u_\infty\in E_\infty$ be a
nondegenerate solution of \eqref{eq:limit}. Then there exists $\La\ge1$ 
and a $\cC^1$-function
\[
[\La,\infty)\to E,\quad u\mapsto u_\la,
\]
such that $u_\la$ is a solution of \eqref{eq:NLS}, and $u_\la\to u_\infty$
as $\la\to\infty$.
\end{Thm}

Problem \eqref{eq:NLS} has found much interest in recent years after 
being first considered in 
\cite{bartsch-wang:1995, bartsch-pankov-wang:2001}. Most papers deal
with potentials being positive and bounded away from $0$, i.~e.\ 
$\inf a_0>0$, exceptions being \cite{bartsch-tang:2012, ding-szulkin:2007}. 
The equation \eqref{eq:NLS} with asymptotically linear nonlinearity has 
been studied in 
\cite{liu-huang:2009,liu-huang-liu:2011,stuart-zhou:2006,wang-zhou:2009}, 
with critical growth nonlinearity in 
\cite{alves-deMorais-souto:2009,alves-souto:2008}, with Neumann boundary 
conditions in exterior domains in \cite{chabrowski-wang:2007}. In  
\cite{bartsch-tang:2012,ding-tanaka:2003, sato-tanaka:2009} 
multiplicity results have been obtained provided the bottom $\overline\Om$
of the potential well consists of several connected components. 
Extensions to quasilinear problems can be found in \cite{alves:2006},
to the Schr\"odinger-Poisson system in \cite{jiang-zhou:2011}. \\

In almost all earlier papers on the topic the authors made assumptions on 
$a,a_0,f$ such that variational methods (e.~g.\ the mountain pass theorem 
or some linking theorem) can be applied to show that \eqref{eq:NLS} has a 
solution $u_\la$. Then it is proved that $u_\la$ converges as 
$\la \to \infty$ towards a solution $u_\infty$ of the limit problem 
\eqref{eq:limit}. However, the limit $u_\infty$ has not been prescribed in 
these papers as we do here. A notable exception, and the only one we are
aware of, where the limit has been prescribed is 
\cite[Theorem~1.2]{sato-tanaka:2009}. There the authors considered the 
one-dimensional problem
\begin{equation}\label{eq:Sako-Tanaka}
-u''+(1+\la a(x))u = |u|^{p-1}u, \quad u\in H^1(\R),
\end{equation}
with the limit problem
\begin{equation}\label{eq:Sako-Tanaka_limit}
\left\{
\begin{aligned}
-u''+u &= |u|^{p-1}u, \quad x\in\Om=(a_1,b_1)\cup(a_2,b_2),\\
u(a_i) &= u(b_i) = 0.
\end{aligned}
\right.
\end{equation}
The solutions of \eqref{eq:Sako-Tanaka_limit} can be listed as $v_{i,j}$, 
$i,j\in\Z$, where $v_{\pm i,\pm j}$ are the unique solutions having $|i|$ 
zeroes in $(a_1,b_1)$ and $|j|$ zeroes in $(a_2,b_2)$. The authors find 
solutions $u_\la$ of \eqref{eq:Sako-Tanaka} such that $u_\la\to v_{i,j}$ as 
$\la\to\infty$. The proof is based on ODE methods and cannot be extended 
to dimensions $N\ge2$. It depends on the uniqueness of the solutions 
having a certain nodal structure. Observe that in the one-dimensional 
case the solutions $v_{i,j}$ are automatically non-degenerate, hence our 
Theorem~\ref{thm:nondeg1} applies. Thus we improve and generalize
\cite[Theorem~1.2]{sato-tanaka:2009} considerably. Moreover, our proof is 
simpler and extends to the PDE setting.\\

In contrast to all earlier papers we do not require global linking type 
hypotheses. Our results may be considered as a local version of these 
earlier results. As a consequence, we can deal with solutions of 
\eqref{eq:limit} which are obtained not using a global linking structure. 
This implies in particular to almost critical problems like
\[
-\Delta u = |u|^{2^*-2-\eps} u,\quad u\in H^1_0(\Om),
\]
where in the limit for $\eps\to0$ the problem can be reduced via the 
Lyapunov-Schmidt reduction method to finding critical points of a 
finite-dimensional limit function; see
\cite{bartsch-d'aprile-pistoia:2012a,bartsch-d'aprile-pistoia:2012b,
bartsch-micheletti-pistoia:2006,musso-pistoia:2010, pistoia-weth:2007}.
For instance, in \cite{bartsch-d'aprile-pistoia:2012b} the solutions 
have been obtained by finding a local minimum and a local mountain
pass of the reduced functional. 

%
\section{Critical points near a singular limit}\label{sec:crit-point}
Let $E$ be a real Hilbert space with scalar product 
$\langle\cdot,\cdot\rangle$, and let $A:E\to E$ be a bounded 
self-adjoint linear operator. We require that $A\ge0$ and that 
$E_\infty:=\ker A\ne\{0\}$. Finally, let $K:E\to\R$ be a 
$C^1$-function, and set $k:=\nabla K: E\to E$. \\

We are interested in finding critical points of the functional
\[
 J_\la: E\to\R, \quad
 J_\la(u):= \frac12\|u\|^2 + \frac{\la}{2}\langle Au,u\rangle - K(u)
\]
for $\la$ large. Observe that $J_\la(u)$ is independent of $\la$ for
$u\in E_\infty$. Moreover, for $u\in E\setminus E_\infty$ we have 
$J_\la(u)\to\infty$ as $\la\to\infty$. We set $K_\infty=K|E_\infty$,
$k_\infty:=\nabla K_\infty:E_\infty\to E_\infty$, and consider
\[
J_\infty:E_\infty\to\R,\quad J_\infty(u) = \frac12\|u\|^2 - K_\infty(u),
\]
as singular limit functional. Clearly, $J_\infty$ is just the restriction of 
$J_\la$ to $E_\infty$. \\

Observe that $\langle Au,u\rangle>0$ for $u\in E\setminus E_\infty$ and that
\begin{equation}\label{eq:prop-A}
u_n \weakto u \text{ in }E,\ \langle Au_n,u_n\rangle\to0
\quad\Longrightarrow\quad Au_n\to0,\ u\in E_\infty.
\end{equation}
This can be seen by looking at the symmetric positive-semidefinite bilinear
form $(u,v)_A := \langle Au,v\rangle$. The Schwarz inequality yields
\[
\|Au\|^2 = (u,Au)_A \le \sqrt{(u,u)_A}\sqrt{(Au,Au)_A}.
\]
Therefore $\langle Au,u\rangle = (u,u)_A = 0$ implies $Au=0$. Similarly,
$u_n \weakto u$, $\langle Au_n,u_n\rangle \to 0$ implies $Au_n\to0=Au$.\\

For $\la\ge0$ and $u,v\in E$ we define
\[
\langle u,v\rangle_\la := \langle u,v\rangle + \la \langle Au,v\rangle.
\]
As a consequence of our hypotheses on $A$ this is a scalar product on $E$, and
it defines a norm $\|\,\cdot\,\|_\la$ on $E$ which is equivalent to the given 
norm corresponding to $\la = 0$. Observe that the orthogonal complement of 
$E_\infty$ with respect to $\langle \cdot,\cdot\rangle_\la$,
\[
E_\infty^\perp = \{u\in E:\langle u,v\rangle_\la = 0 \text{ for all }v\in E_\infty\}
\]
is independent of $\la$, hence the orthogonal projections 
$P:E\to E_\infty$ and $Q=\id - P: E\to E_\infty^\perp$ are independent of $\la$. 
We write $B_{r,\la}(0,E_\infty^\perp) := \{v\in E_\infty^\perp:\|v\|_\la\le r\}$ 
for the ball of radius $r>0$ around $0\in E_\infty^\perp$ with respect to the 
norm $\|\,\cdot\,\|_\la$. For $\de>0$ and $\la>0$ and $u\in E_\infty$ we define
\[
B_{\de,\la}(u):=B_{\de}(u,E_\infty)\times B_{\de,\la}(0,E_\infty^\perp)\subset E.
\]

Given a bounded linear map $L: E\to E$ we write
\[
\|L\|_\la := \sup\{ \|Lu\|_\la: u\in E,\ \|u\|_\la\le 1 \}
\]
for the operator norm of $L$ with respect to $\|\,\cdot\,\|_\la$ on $E$.\\

For $\la>0$ we define the nonlinear operators 
$k_\la = \nabla_\la K: E\to E$ and $\nabla_\la J_\la: E\to E$ by the equations 
\[
\langle k_\la(u),v\rangle_\la = \langle \nabla_\la K(u),v\rangle_\la
 = DK(u)[v] = \langle k(u),v\rangle,
\]
and $\nabla_\la J_\la = \id-k_\la$. Observe that
\begin{equation}\label{eq:k_lambda}
\|k_\la(u)\|_\la = \sup_{\|v\|_\la\le1}\langle k_\la(u),v\rangle_\la
 = \sup_{\|v\|_\la\le1}\langle k(u),v\rangle 
 \le \sup_{\|v\|\le1}\langle k(u),v\rangle
 = \|k(u)\|\,.
\end{equation}

If $K$ is of class $C^2$ near $u$ then the derivatives of $k=\nabla K$ and of
$k_\la=\nabla_\la K$ satisfy
\begin{equation}\label{eq:derivatives}
\langle Dk_\la(u)[v],w\rangle_\la = \langle Dk(u)[v],w\rangle \quad
\text{for } u,v,w\in E.
\end{equation}

We also deduce for $\la\ge0$ and $u\in E$ that
\[
\begin{aligned}
\| Dk_\la(u)\|_\la^2
& = \sup_{\|v\|_\la\le1}\langle Dk_\la(u)[v],Dk_\la(u)[v]\rangle_\la
  = \sup_{\|v\|_\la\le1}\langle Dk(u)[v],Dk_\la(u)[v]\rangle \\
& \le \sup_{\|v\|_\la\le1} \|Dk(u)\| \cdot \|v\| \cdot \|Dk_\la(u)\|_\la
       \cdot \|v\|_\la
  \le \|Dk(u)\| \cdot \|Dk_\la(u)\|_\la
\end{aligned}
\]
hence,
\begin{equation}\label{eq:norms}
\| Dk_\la(u)\|_\la \le \|Dk(u)\|.
\end{equation}
Similarly we obtain for $\la\ge0$ and $u,v\in E$ that
\begin{equation}\label{eq:norms2}
\| Dk_\la(u) - Dk_\la(v) \|_\la \le \| Dk(u) - Dk(v) \|.
\end{equation}

Now we collect some hypotheses on $J_\la$ which we will impose in the 
various results.
\begin{itemize}
\item[$(J_1)$] $J_\infty$ has an isolated critical point $u_\infty\in E_\infty$, 
and the critical groups of $u_\infty$ as a critical point of $J_\infty$ are 
nontrivial: 
$\cC_*(J_\infty,u_\infty)\ne0$.
\end{itemize}

We fix $\de_0>0$ such that $u_\infty$ is the only critical point of $J_\infty$
in $B_{\de_0}(u_\infty)$.

\begin{itemize}
\item[$(J_2)$] There exists $\la_0>0$ such that $k$ is weakly 
sequentially continuous in $B_{\de_0,\la_0}(u_\infty)$, i.~e.\ if
$u_n\in B_{\de_0,\la_0}(u_\infty)$ and $u_n \weakto u$ then $k(u_n) \weakto k(u)$.
\end{itemize}

\begin{itemize}
\item[$(J_3)$] 
$\displaystyle
B_{\de_0,\la_n}(u_\infty)\ni u_n\weakto u,\ \la_n\to\infty
 \qquad\Longrightarrow\qquad k(u_n)\to k(u).
$
\end{itemize}

\begin{itemize}
\item[$(J_4)$] There exists $\la_0>0$ such that $k$ is completely continuous 
in $B_{\de_0,\la_0}(u_\infty)$.
\end{itemize}

Condition $(J_2)$ is rather harmless, also $(J_3)$ holds under rather general
assumptions on $a,a_0$, and $f$. Both are much weaker than requiring that $k$ 
is completely continuous near $u_\infty$ as in $(J_4)$. $(J_2)$ does imply that 
$k_\infty$ is completely continuous in $B_{\de_0}(u_\infty,E_\infty)$. Therefore 
$J_\infty$ satisfies the Palais-Smale condition in $B_{\de_0}(u_\infty,E_\infty)$, 
i.~e.\ any Palais-Smale sequence $u_n\in B_{\de_0}(u_\infty,E_\infty)$ for $J_\infty$ 
has a convergent subsequence.

\begin{Thm}\label{thm:crit-group2}
Suppose that $(J_1)$,  $(J_2)$, and $(J_3)$ hold. Then there exists 
$\La\ge0$ such that $J_\la$ has a critical point $u_\la$ for 
$\la\in[\La,\infty)$ and such that $u_\la\to u_\infty$ as $\la\to\infty$.
\end{Thm}

Our next result is based on degree theory. Recall that the index of $u_\infty$ 
as fixed point of $k_\infty$ is defined as:
\[
\ind(k_\infty,u_\infty)
 = \deg(\id-k_\infty, B_\de(u_\infty,E_\infty),u_\infty)
\]
where $\deg$ denotes the Leray-Schauder degree, $0<\de\le\de_0$. This index 
is defined, for instance, if $k_\infty$ is completely continuous in 
$B_{\de_0}(u_\infty,E_\infty)$, hence if $(J_2)$ or $(J_3)$ holds, in particular 
if $(J_4)$ holds. 

\begin{Thm}\label{thm:index2}
Suppose $(J_1)$ and $(J_4)$ hold. Suppose moreover that the local 
fixed point index of $u_\infty$ as a fixed point of $k_\infty$ is nontrivial:
\[
\ind(k_\infty,u_\infty)
 = \deg(\id-k_\infty, U_\eps(u_\infty,E_\infty),u_\infty) \ne 0.
\]
Then there exists a connected set $\cS \subset [\La,\infty) \times E$ covering
the parameter interval $[\La,\infty)$ for some $\La\ge 1$, such that
$\nabla J_\la(u) = 0$ for every $(\la,u)\in\cS$. Moreover, given a sequence
$(\la_n,u_n) \in \cS$ with $\la_n \to \infty$ there holds $u_n \to u_\infty$.
\end{Thm}

\begin{Rem}\label{rem:gen-degree}
{\rm a)} Recall that under the conditions of Theorem~\ref{thm:crit-group2} the 
local index $\ind(K_\infty,u_\infty)$ may be trivial. On the other hand, if the 
local index $\ind(K_\infty,u_\infty)$ is nontrivial, then the critical groups
$\cC_*(f,u_\infty)$ are nontrivial. This follows from the Poincar\'e-Hopf
formula~\eqref{eq:Poincare-Hopf}.

{\rm b)} Assumption $(J_4)$ can be replaced by any assumption assuring that 
there is a degree theory for the maps $\id-k_\la$. In the case of $(J_4)$ one
has the Leray-Schauder degree. If $k_\la$ is, for instance, A-proper in the
sense of \cite{Petryshin:1995}, the generalized degree of Petryshin can be 
applied; see Theorem \ref{thm:set-contraction} below and its proof in Section 
\ref{sec:degree}.
\end{Rem}

Surprisingly, the compactness condition can be considerable relaxed if $K$ 
is $C^2$ near $u_\infty$. We need the following condition on the differential 
$Dk(u_\infty)$.

\begin{itemize}
\item[$(J_5)$] If $u_n\weakto u$ and $\|u_n\|_{\la_n}$ is bounded for some
sequence $\la_n\to\infty$, then $Dk(u_\infty)[u_n]\to Dk(u_\infty)[u]$.
\end{itemize}

\begin{Thm}\label{thm:set-contraction}
Suppose $K$ is $C^2$ near $u_\infty$, $(J_1)$, $(J_3)$, and $(J_5)$ are 
satisfied. Then the conclusion of Theorem \ref{thm:index2} holds true.
\end{Thm}

Under the assumptions of Theorem \ref{thm:set-contraction} $k_\la$ need not
be compact, so we cannot work with the Leray-Schauder degree. Instead we 
will be able to use the degree for $\be$-condensing maps where where $\be$ 
is the ball measure of noncompactness; see \cite{Deimling:1985}. \\

Finally we state a result in the nondegenerate setting. 

\begin{Thm}\label{thm:nondeg2}
Suppose that $K$ is $C^2$ near $u_\infty$, that $u_\infty$ is a nondegenerate 
critical point of $J_\infty$, and that $(J_5)$ is satisfied. Then there exists 
$\La \ge 0$ and a $C^1$-map
$[\La,\infty) \to E$, $\la\mapsto u_\la$, such that $u_\la$ is the unique
critical point of $J_\la$ near $u_\infty$ for $\la \in [\La,\infty)$. Moreover,
$u_\la \to u_\infty$ as $\la \to \infty$.
\end{Thm}

\section{Nontrivial critical groups}\label{sec:variational}
We first prove Theorem~\ref{thm:crit-group2}. Consider the isolated critical 
point $u_\infty\in E_\infty$ of $J_\infty$ with nontrivial critical groups. Let 
$(W,W_-)$ be a Gromoll-Meyer pair for $u_\infty$ in $B_{\de_0}(u_\infty,E_\infty)$. 
This means that:

\begin{itemize}
\settowidth{\itemindent}{7pt}

\item[$(GM_1)$]
 $W\subset\opint B_{\de_0}(u_\infty,E_\infty)$ is a closed neighborhood of $u_\infty$ 
 in $E_\infty$ containing no other critical point of $J_\infty$.
\item[$(GM_2)$]
 There exist $C^1$-functions $g_i:U\cap E_\infty\to\R$, $i=1,\dots,l$, having
 $0$ as regular value, such that $W=\bigcap_{i=1}^l g_i^0$ and
 $\pa W = W\cap \bigcup_{i=1}^l g_i^{-1}(0)$.
\item[$(GM_3)$]
 $\nabla J_\infty$ is transversal to each $g_i^{-1}(0)$; more precisely, for
 some $\al>0$:
 \[
 \langle \nabla J_\infty(u),\nabla g_i(u)\rangle \le -3\al < 0\quad
  \text{for }u\in \pa W\cap g_i^{-1}(0),\  i=1,\dots,j,
 \]
 and
 \[
 \langle \nabla J_\infty(u),\nabla g_i(u)\rangle \ge 3\al > 0\quad
  \text{for }u\in \pa W\cap g_i^{-1}(0),\ i=j+1,\dots,l.
 \]
\item[$(GM_4)$] The exit set
\[
\begin{aligned}
  W_-&=\left\{u\in\pa W: \langle \nabla J_\infty(u),\nabla g_i(u)\rangle < 0
       \text{ if } u\in  g_i^{-1}(0),i=1,...,l\right\}\\
     &= \bigcup_{i=1}^j g_i^{-1}(0)
\end{aligned}
\]
consists of those $x\in \pa W$ where $-\nabla J_\infty$ points outside of $W$.
\end{itemize}

A construction of a Gromoll-Meyer pair can be found in \cite[p.~49]{chang:1993},
where $l=3$, $j=1$. Using a pseudo-gradient vector field for $J_\infty$ it is
standard to show that
\begin{equation}\label{eq:gromoll}
C_*(J_\infty,u_\infty) = H_*(J^c,J^c\setminus\{u_\infty\}) \cong H_*(W,W_-) .
\end{equation}
This uses that $J_\infty$ satisfies the Palais-Smale condition in 
$B_{\de_0}(u_\infty,E_\infty)$. \\

For $\de>0$ we set
\[
W_{\de,\la}:=W\times B_{\de,\la}(0,E_\infty^\perp).
\] 

\begin{Lem}\label{lem:deg-est}
If $(J_3)$ holds then for every $\eps>0$ there exists $\La>0$ such that
\[
\sup_{u\in W_{\de_0,\la}} \|k_\la(u)-k_\infty(Pu)\|_\la \le \eps \qquad
\text{for all }\la\ge\La.
\]
\end{Lem}

\begin{proof}
Arguing by contradiction, suppose there exist $\eps>0$, $\la_n\to\infty$, 
$u_n\in W_{\de_0,\la_n}$ with
\begin{equation*}
\|k_{\la_n}(u_n)-k_\infty(Pu_n)\|_{\la_n} \geq \eps.
\end{equation*}
Then the sequence $(u_n)_n$ is bounded, and $\langle Au_n,u_n\rangle \to 0$, 
hence \eqref{eq:prop-A} applies and yields $u_n,\ Pu_n\weakto u\in E_\infty$ 
along a subsequence. Now $(J_3)$ implies $k(u_n)\to k(u)$ and 
$k(Pu_n)\to k(u)$. Setting $v_n:=k_{\la_n}(u_n)-k_\infty(Pu_n)$ we see that 
$\|v_n\|_{\la_n}$ is bounded uniformly in $n$ as a consequence of 
\eqref{eq:k_lambda}. Applying \eqref{eq:prop-A} again shows that 
$v_n,Pv_n \weakto v\in E_\infty$ along a subsequence. This in turn implies
\[
\begin{aligned}
\eps^2 &\le \|v_n\|_{\la_n}^2
  = \langle k_{\la_n}(u_n),v_n\rangle_{\la_n}-\langle k_{\infty}(Pu_n),v_n\rangle\\
 &= \langle k(u_n),v_n\rangle-\langle k(Pu_n),Pv_n\rangle
  \to \langle k(u),v\rangle-\langle k(u),v\rangle
  = 0
\end{aligned}
\]
which is absurd. 
\end{proof}

\begin{Lem}\label{lem:gromoll}
For all $0<\de\le\de_0$ there exists $\La_\de>0$ such that for $\la\ge\La_\de$ 
and $v\in B_{\de,\la}(0,E_\infty^\perp)$, there holds:
\[
\langle P\nabla_\la J_\la(u+v),\nabla g_i(u)\rangle \le -2\al \quad\text{for }
 u\in  \pa W\cap g_i^{-1}(0),\ i=1,\dots,j,
\]
and
\[
\langle P\nabla_\la J_\la(u+v),\nabla g_i(u)\rangle \ge 2\al \quad\text{for }
 u\in  \pa W\cap g_i^{-1}(0),\ i=j+1,\dots,l,
\]
and
\[
\langle Q\nabla_\la J_\la(u+v),v\rangle_\la \ge \de^2/2 \quad\text{for }
 u\in W,\ v\in B_{\de,\la}(0,E_\infty^\perp).
\]
\end{Lem}

\begin{proof}
We may assume that
\begin{equation*}
m:=\max_{i=1,\dots,l}\sup_{u\in g_i^{-1}(0)\cap W} \|\nabla g_i(u)\|<\infty.
\end{equation*}
According to Lemma \ref{lem:deg-est}, for $0<\de\le\de_0$ there exists 
$\La_\de>0$ such that 
\[
\|k_\la(u)-k_\infty(Pu)\|_\la \leq \min\{\alpha/m,\de/2\} \quad
\text{if }\la\ge\La_\de,\ u\in W_{\de_0,\la}.
\]
Consequently, we obtain for $\la\geq\La_\de$, $i=1,\dots,j$, 
$u \in W\cap g^{-1}_i(0)$ and $v\in B_{\de,\la}(0,E_\infty^\perp)$ that
\[
\begin{aligned}
&\langle P\nabla_\la J_\la(u+v),\nabla g_i(u)\rangle\\
&\hspace{1cm}
 \leq \langle \nabla J_\infty(u),\nabla g_i(u)\rangle 
        + \|P\nabla_\la J_\la (u+v)-\nabla J_\infty(u)\|\cdot\|\nabla g_i(u)\|\\
&\hspace{1cm}
 = \langle \nabla J_\infty(u),\nabla g_i(u)\rangle
        + \|Pk_\la (u+v)-k_\infty(u)\|\cdot\|\nabla g_i(u)\|\\
&\hspace{1cm}
 \leq -3\alpha+\frac{\alpha}{m}m=-2\alpha.
\end{aligned}
\]
Similarly we obtain for $\la\geq\La_\de$, $i=j+1,\dots,l$, 
$u \in W\cap g^{-1}_i(0)$ and $v\in B_{\de,\la}(0,E_\infty^\perp)$ that
\begin{equation*}
\langle P\nabla_\la J_\la(u+v),\nabla g_i(u)\rangle\geq 2\alpha.
\end{equation*}
Finally, for $\la\geq\La_\de$, $u \in W$ and $v\in B_{\de,\la}(0,E_\infty^\perp)$ 
there holds:
\[
\begin{aligned}
\langle \nabla_\la QJ_\la (u+v),v\rangle_\la
 &=\|v\|_\la^2 - \langle k_\la(u+v)-k_\infty(u),v\rangle_\la
  \geq \|v\|_\la^2- \|\langle k_\la(u+v)-k_\infty(u)\|_\la\|v\|_\la\\
 &\geq \de^2 - \frac{\de}{2}\de = \frac{\de^2}{2}.
\end{aligned}
\]
\end{proof}

Lemma \ref{lem:gromoll} implies that for $\de>0$ and $\la\ge\La_\de$, the set
\[\left(W_{\de,\la}, W_-\times B_{\de,\la}(0,E_\infty^\perp)\right)
 =\left(W\times B_{\de,\la}(0,E_\infty^\perp),
   W_-\times B_{\de,\la}(0,E_\infty^\perp)\right)
\]
is a regular index pair for pseudo-gradient flows of $J_\la$ in the sense of
Conley index theory.\\

\begin{Lem}\label{lem:ex-crit-pt}
$J_\la$ has a critical point $u_\la\in W_{\de,\la}$ if $0<\de\le\de_0$ and 
$\la\ge\La_\de$.
\end{Lem}

\begin{proof}
If $J_\la$ does not have a critical point in $W_{\de,\la}$ then there exists 
a pseudo-gradient vector field $V$ for $J_\la$ in $W_{\de,\la}$ such that the 
inequalities in Lemma~\ref{lem:gromoll} hold with $V$ instead of 
$\nabla_\la J_\la$, $\al$ instead of $2\al$, and $\de^2/4$ instead of $\de^2/2$. 
Moreover, 
\begin{equation}\label{eq:lower-bound}
\inf_{u\in W_{\de,\la}}\|\nabla_\la J_\la(u)\|_\la > 0,
\end{equation}
because if $u_n\in W_{\de,\la}$ satisfies 
$\nabla_\la J_\la(u_n)=u_n-k_\la(u_n) \to 0$,
then $u_n\weakto u\in W_{\de,\la}$ along a subsequence, hence 
$k_\la(u_n)\weakto k_\la(u)$ as a consequence of $(J_2)$. This implies that 
$u\in W_{\de,\la}$ is a critical point of $J_\la$.
Observe that we do not prove strong convergence here, hence we do not prove the 
Palais-Smale condition in $W_{\de,\la}$.\\ 

Now \eqref{eq:lower-bound} implies that the flow associated to $-V$ provides a 
deformation of $W_{\de,\la}$ to $W_-\times B_{\de,\la}(0,E_\infty^\perp)$. This in 
turn implies
\[
H_*\left(W_{\de,\la},W_-\times B_{\de,\la}(0,E_\infty^\perp)\right) \cong 0
\]
in contradiction with
\[
H_*\left(W_{\de,\la},W_-\times B_{\de,\la}(0,E_\infty^\perp)\right)
\cong H_*(W,W_-) \cong C_*(J_\infty,u_\infty) \ne 0.
\]
\end{proof}

\begin{altproof}{\ref{thm:crit-group2}}
The existence of a critical point
$u_\la\in W\times B_{\de,\la}(0,E_\infty^\perp)$ of $J_\la$
for $\la\ge\La$ has been stated in Lemma~\ref{lem:ex-crit-pt}. Clearly,
$Qu_\la \to 0$ and $\nabla J_\infty(Pu_\la) \to 0$ as $\la\to\infty$.
It follows that $u_\la \to u_\infty \in E_\infty$ because $u_\infty$ is the only 
critical point of $J_\infty$ in $W$.
\end{altproof}

\begin{altproof}{\ref{thm:crit-group1}}
In order to apply Theorem \ref{thm:crit-group2} we set
\[
E := \left\{u\in H^1(\R^N): \int_{\R^N} a_0u^2 < \infty,\ 
            \int_{\R^N} au^2 < \infty \right\}
\]
provided with the scalar product
\[
\langle u,v\rangle := \int_{\R^N} (\nabla u,\nabla v +(b+a_0+a)uv)dx
\]
Here $b=1-\essinf\,a_0$ is defined as in Section~\ref{sec:NLS}. The operator 
$A:E\to E$ is defined by the equation
\[
\langle Au,v\rangle := \int_{\R^N} auvdx,
\]
and the functional $K:E\to\R$ by
\[
K(u) = -\int_{\R^N} \left(\frac{b}{2}u^2+F(x,u)\right)\,dx.
\]
$A$ is a self-adjoint, positive semidefinite, and bounded linear operator.
The kernel $E_\infty$ of $A$ consists of all $u\in E$ such that $u=0$ a.~e.\ 
in $\R^N\setminus\Om$, hence $E_\infty = H^1_0(\Om)$. This uses that the
boundary of $\Om$ is Lipschitz. \\

Solutions of \eqref{eq:NLS} are obtained as critical 
points of the $\cC^1$-functional
\[
\begin{aligned}
 J_{\la}(u)
  &= \frac12\int_{\R^N} (|\nabla u|^2+(a_0+\la a)u^2\,dx - \int_{\R^N} F(x,u)\,dx\\
  &= \frac12\left(\|u\|^2 + (\la-1)\langle Au,u\rangle\right) - K(u).
\end{aligned}
\]
Observe that $\la$ has to be replaced by $\la-1$ because $\|\cdot\|_\la$ 
contains the summand $\langle Au,u\rangle$. Since $a=0$ on $\Om$, 
we see that $J_\infty$ is simply the restriction of $J_\la$ to $E_\infty$.\\

It remains to prove the conditions $(J_2)$ and $(J_3)$. In fact, $(J_2)$ 
is an easy consequence of $(f_1)$ because $E$ imbeds into $L^p(\R^N)$ for 
$2\le p\le 2^*$. In order to see $(J_3)$, consider sequences $\la_n\to\infty$ 
and $B_{\de_0,\la_n}\ni u_n\weakto u$. Then $u_n \to u$ strongly in $L^p(\R^N)$ 
for $2<p<2^*$ by \cite[Lemma~4.2.]{bartsch-pankov-wang:2001}. And $u_n \to u$ 
strongly in $L^2(\R^N)$ follows from $(V_3)$. This implies $k(u_n)\to k(u)$ 
in $E$ because of the subcritical growth of $f$ required in $(f_1)$. \\

Theorem~\ref{thm:crit-group1} is now an immediate consequence of
Theorem~\ref{thm:crit-group2}.
\end{altproof}

\section{Nontrivial index}\label{sec:degree}
\begin{altproof}{Theorem~\ref{thm:index2}} From 
$\|u-k_\infty(Pu)\| \ge \|Pu-k_\infty(Pu)\|$ and using Lemma \ref{lem:deg-est}
we immediately deduce that there exists $\La>0$ such that for $\la\ge\La$ 
and $0<\de\le\de_0$ small that
\begin{equation}\label{eq:degree1}
\begin{aligned}
0&\ne\ind(k_\infty,u_\infty)
  = \deg(\id_{E_\infty}-k_\infty, B_{\de}(u_\infty,E_\infty),u_\infty)\\
 & = \deg(\id_E-k_\infty\circ P,B_{\de,\la}(u_\infty),u_\infty)\\
 & = \deg(\id_E-k_\la,B_{\de,\la}(u_\infty),u_\infty)\,.
\end{aligned}
\end{equation}

Since $k$ is completely continuous in $B_{\de_0,\la_0}(u_\infty)$ so is $k_\la$, 
hence the above degree is defined

Using \eqref{eq:degree1}, a standard continuation argument (see
\cite{Alexander:1981}, for instance) shows that there exists a connected set
$\cS\subset[\La,\infty)\times B_{\de,\la}(u_\infty)\subset[\La,\infty)\times E$ 
covering the parameter inverval $[\La,\infty)$, such that 
$\nabla J_\la(u) = u-k_\la(u) = 0$ for every $(\la,u)\in\cS$. Given a 
sequence $(\la_n,u_n) \in \cS$ with $\la_n \to \infty$, using
\eqref{eq:degree1} and Lemma~\ref{lem:deg-est} once more, we deduce that
$\|u_n-u_\infty\|_\la \to 0$.
\end{altproof}

\begin{altproof}{Theorem \ref{thm:index1}}
This follows from Theorem~\ref{thm:index2} as Theorem~\ref{thm:crit-group1}
follows from Theorem~\ref{thm:crit-group2}. We only need to observe that
$k$ is completely continuous as a consequence of $(V_5)$ and $(f_1)$, in
particular $(J_4)$ is satisfied.
\end{altproof}

For the proof of Theorem \ref{thm:set-contraction} we need the following 
lemma.

\begin{Lem}\label{lem:contraction1}
Suppose $(J_5)$ is satisfied. Then
\[
\| Dk_\la(u_\infty) - Dk_\infty(u_\infty)\circ P \|_\la \to 0 
 \quad\text{as $\la\to\infty$.}
\]
\end{Lem}

\begin{proof}
Arguing by contradiction, suppose that there exist sequences 
$\la_n\to\infty$, $u_n\in E$ with $\|u_n\|_{\la_n}=1$, and
\begin{equation*}
 \|Dk_{\la_n}(u_\infty)[u_n]-Dk_\infty(u_\infty)[Pu_n]\|_\la\geq \eps>0.
\end{equation*}
Then $u_n\weakto u$ in $E$ along a subsequence, and $u\in E_\infty$ by 
\eqref{eq:prop-A}, hence also $Pu_n \weakto u$.
Setting
\begin{equation*}
    v_n:=Dk_{\la_n}(u_\infty)[u_n]-Dk_\infty(u_\infty)[Pu_n]
\end{equation*}
and using \eqref{eq:norms} we see that
\begin{equation*}
\|v_n\|_{\la_n} \le \|Dk_{\la_n}(u_\infty)[u_n]\|_{\la_n}  + \|Dk(u_\infty)[Pu_n]\|
\le \|Dk(u_\infty)\| + \|Dk(u_\infty)\|
\end{equation*}
is bounded uniformly in $n$. We deduce, again by \eqref{eq:prop-A}, that 
$v_n\weakto v$ in $E$ along a subsequence, and that $v\in E_\infty$, hence also 
$Pv_n\weakto v$. Using condition $(J_5)$ we obtain a contradiction:
\begin{equation*}
\begin{aligned}
\eps^2 & \leq \|v_n\|_{\la_n}^2
   =\langle Dk_{\la_n}(u_\infty)[u_n], v_n\rangle_{\la_n}
     - \langle Dk_\infty(u_\infty)[Pu_n], Pv_n\rangle\\
 & = \langle Dk(u_\infty)[u_n],v_n\rangle-\langle Dk(u_\infty)[Pu_n],Pv_n\rangle\\
 & \to \langle Dk(u_\infty)[u],v\rangle-\langle Dk(u_\infty)[u],v\rangle=0.
\end{aligned}
\end{equation*}
\end{proof}

\begin{altproof}{Theorem \ref{thm:set-contraction}}. Let $\be_\la$ be the ball 
measure of non-compactness in $E$, i.~e.\ for a subset $A\subset E$
\[
\be_\la(A) =\inf\{r>0: \text{$A$ can be covered by finitely many
                  $\|\cdot\|_\la$-balls of radius $r$}\}.
\]
We claim that $k_\la$ is a strict $\be_\la$-set contraction in a neighborhood
of $u_\infty$ if $\la$ is large. We refer to \cite{Deimling:1985} for 
properties of this class of maps and the construction of a degree theory. It 
is sufficient to show that 
\begin{equation}\label{eq:set-contraction}
\be_\la(k_\la(A))\le\frac12\be_\la(A) \quad
\text{for $A\subset B_{\de,\la}(u_\infty)$ if $\la$ is large and $\de$ is small.}
\end{equation}
For \eqref{eq:set-contraction} it suffices to prove that $k_\la-k_\infty\circ P$ 
is $\|\cdot\|_\la$-Lipschitz continuous with Lipschitz constant $\frac12$ 
because $k_\infty\circ P$ is completely continuous as a consequence of $(J_3)$, 
and because the sum of a completely continuous map and a Lipschitz map with 
Lipschitz constant $\frac12$ satisfies \eqref{eq:set-contraction}. Now the 
Lipschitz continuity of $k_\la-k_\infty\circ P$ follows easily from:
\[
\begin{aligned}
&\|k_\la(u)-k_\infty(Pu)-(k_\la(v)-k_\infty(Pv))\|_\la\\
&\hspace{1cm}
 \le \|k_\la(u)-k_\la(v)-Dk_\la(u_\infty)[u-v])\|_\la\\
&\hspace{2cm}
 + \|Dk_\la(u_\infty)[u-v]-Dk_\infty(u_\infty)[Pu-Pv]\|_\la\\
&\hspace{2cm}
 + \|k_\infty(Pu)-k_\infty(Pv)-Dk_\infty(u_\infty)[Pu-Pv]\|_\la\\
&\hspace{1cm}
 \le \sup_{w\in B_{\de,\la}(u_\infty)}\|Dk_\la(w)[u-v]-Dk_\la(u_\infty)[u-v]\|_\la\\
&\hspace{2cm}
 + \|Dk_\la(u_\infty)-Dk_\infty(u_\infty)\circ P\|_\la\|u-v\|_\la\\
&\hspace{2cm}
 + \sup_{w\in B_{\de}(u_\infty,E_\infty)}\|Dk_\infty(w)[u-v]-Dk_\infty(u_\infty)[u-v]\|
\end{aligned}
\]
Now $\sup_{w\in B_{\de,\la}(u_\infty)}\|Dk_\la(w)-Dk_\la(u_\infty)\|_\la$ and
$\sup_{w\in B_{\de}(u_\infty,E_\infty)}\|Dk_\infty(w)-Dk_\infty(u_\infty)$ can be made
arbitrarily small by making $\de>0$ small. And 
$\|Dk_\la(u_\infty)-Dk_\infty(u_\infty)\|_\la\circ P$ can be made arbitrarily small
as $\la\to\infty$ as a consequence of Lemma \ref{lem:contraction1}.\\

Since $k_\la$ is a strict $\be_\la$-set contraction in a neighborhood
of $u_\infty$ for $\la$ large, we may argue as in the proof of
Theorem~\ref{thm:index2} to conclude the proof of 
Theorem~\ref{thm:set-contraction}.
\end{altproof}

\begin{altproof}{Theorem \ref{thm:index1a}}
Observe that $(f_1')$ implies that 
\[
K:E\to\R,\quad K(u)=\int_{\R^N} \left(\frac{b}{2}u^2+F(x,u)\right)\,dx,
\] 
is of class $\cC^2$. It remains to prove $(J_3)$ and $(J_5)$. In fact, the 
proof of $(J_3)$ proceeds as in the proof of Theorem~\ref{thm:crit-group1}. 
In order to see $(J_5)$ consider a sequence $u_n\in E$ such that 
$u_n\weakto u$ and $\|u_n\|_{\la_n}$ is bounded for some sequence 
$\la_n\to\infty$, so that $u\in E_\infty=H^1_0(\Om)$. Now assumption $(V_3)$ 
yields a sequence $R_j\to\infty$ such that
\[
\lim_{n\to\infty}\liminf_{j\to\infty}\frac{\|u_n-u\|_{\la_n}^2}{\int_{K_{R_j}^c}|u_n-u|^2}
 \to \infty,
\]
which implies
\[
\int_{K_{R_j}^c}|u_n-u|^2 \to 0 \quad\text{as }j,n\to\infty.
\]
Since $u_n\to u$ in $L^2_{loc}(\R^N)$ we deduce that $u_n\to u$ in $L^2(\R^N)$.
This implies that
\[
\begin{aligned}
|\langle Dk(u_\infty)[u_n]-Dk(u_\infty)[u],v\rangle| 
 &= \left|\int_{\R^N} (b+f'(u_\infty))(u_n-u)v\,dx\right|\\
 &\le c\|u_n-u\|_{L^2(\Om)}\|v\|
\end{aligned}
\]
hence $Dk(u_\infty)[u_n] \to Dk(u_\infty)[u]$ in $E$.\\

Now Theorem \ref{thm:index1a} follows from Theorem~\ref{thm:set-contraction}.
\end{altproof}

\section{The nondegenerate case}\label{sec:nondeg}
In this section we use the notation $f_\la=\id_E-k_\la:E\to E$. The proof 
of Theorem~\ref{thm:nondeg2} is an immediate consequence of the 
following proposition.

\begin{Prop}\label{prop:contraction}
For $\de>0$ small there exists $\La_\de\ge1$ such that the map
\[
g_\la:B_{\de,\la}(u_\infty) \to B_{\de,\la}(u_\infty),\quad
 g_\la(u) := u-(\id_E-Dk_\infty(u_\infty)\circ P)^{-1}\circ f_\la(u),
\]
is well defined and a contraction for $\la\ge\La_\de$.
\end{Prop}

\begin{altproof}{Theorem~\ref{thm:nondeg2}}
According to Proposition~\ref{prop:contraction} there exists $\de_0>0$ such
that for $0 < \de \le \de_0$ and $\la \ge \La_\de$, the Banach fixed point 
theorem yields a unique fixed point $u_\la \in B_{\de,\la}(u_\infty)$ of $g_\la$, 
hence a zero of $f_\la$, i.~e.\ a critical point of $J_\la$. The map 
\[
[\La_{\de_0},\infty) \to E,\quad \la \mapsto u_\la,
\]
is $C^1$ because $f_\la$ is $C^1$ in $\la$. Finally,
$\|u_\la - u_\infty\|_\la \to 0$ is also a consequence of
Proposition~\ref{prop:contraction}.
\end{altproof}

The proof of Proposition~\ref{prop:contraction} is based on the following 
lemmata.

\begin{Lem}\label{lem:iso}
The bounded operator 
$L:=\id_E-Dk_\infty(u_\infty)\circ P: E\to E$ 
is an isomorphism, and $\|L^{-1}\|_\la \le \al$ is bounded uniformly in $\la$.
\end{Lem}

\begin{proof}
That $L$ is an isomorphism follows easily from the assumption that $u_\infty$ 
is a nondegenerate fixed point of $k_\infty$, which means that
$\id_{E_\infty}-Dk_\infty(u_\infty):E_\infty\to E_\infty$
is an isomorphism. It is also clear that 
$\|L^{-1}\|_\la \le \max\{1, \|u_\infty-Dk_\infty(u_\infty)^{-1}\|\}$
because the norms on $E_\infty$ do not depend on $\la$.
\end{proof}

\begin{Lem}\label{lem:contraction2}
$\| f_\la(u_\infty) \|_\la \to 0$ as $\la \to \infty$.
\end{Lem}

\begin{proof}
Arguing by contradiction, suppose there exist $\eps>0$ and $\la_n\to\infty$
such that 
$v_n:=f_{\la_n}(u_\infty)=u_\infty-k_{\la_n}(u_\infty)
 = k_\infty(u_\infty)-k_{\la_n}(u_\infty)$ satisfies $\|v_n\|_{\la_n}\ge\eps$. 
Observe that $\|v_n\|_{\la_n}$ is bounded uniformly in $n$ as a consequence of
\eqref{eq:k_lambda}. Now \eqref{eq:prop-A} implies 
$v_n,Pv_n \weakto v\in E_\infty$ along a subsequence. This in turn implies:
\[
\eps^2 \le \|v_n\|_{\la_n}^2
 = \langle k_\infty(u_\infty),v_n\rangle
    - \langle k_{\la_n}(u_\infty),v_n\rangle_{\la_n}
 = \langle k(u_\infty),Pv_n\rangle - \langle k(u_\infty),v_n\rangle
 \to 0
\]
which is absurd.
\end{proof}

\begin{altproof}{Proposition~\ref{prop:contraction}}
By \eqref{eq:norms2} there exists $\de_1>0$ such that
\begin{equation}\label{satz2:1}
\sup_{u\in B_{\de_1,\la}(u_\infty)} \|Dk_\la(u)-Dk_\la(u_\infty)\|_{\la}
 \leq\frac{1}{4\alpha} \quad\text{for all }\la\geq 0, 
\end{equation}
where $\al>0$ is from Lemma \ref{lem:iso}.
Now we fix $0<\de\leq \de_1$. Using Lemma \ref{lem:contraction1} and 
Lemma \ref{lem:contraction2} there exists $\La_\de$ such that
\begin{equation}\label{satz2:2}
\|Dk_\la(u_\infty)-Dk_\infty(u_\infty)\circ P\|_\la \leq \frac{1}{4\alpha}
\quad\text{ for }\la\geq\La_\de,
\end{equation}
and
\begin{equation}\label{satz2:3}
\|f_\la(u_\infty)\|_\la \leq \frac{\de}{2\alpha}
\quad\text{ for }\la\geq\La_\de.
\end{equation} 
Thus for $\la \geq \La_\de$ and $u,v\in B_{\de,\la}(u_\infty)$ there holds
\begin{equation}\label{satz2:4}
\begin{aligned}
 &\|k_\la(u)-k_\la(v)-Dk_\la(u_\infty)(u-v)\|_\la\\
 &\hspace{1cm}
  \leq\sup_{w\in B_{\de,\la}(u_\infty)}
       \|Dk_\la(w)-Dk_\la(u_\infty)\|_{\la}\cdot\|u-v\|_\la
 \overset{\eqref{satz2:1}}{\leq}\frac{1}{4\alpha}\|u-v\|_\la.
\end{aligned}
\end{equation}
Since $L=\id_E - Dk_\infty(u_\infty)\circ P$ we have
$g_\la=L^{-1}(k_\la-Dk_\infty(u_\infty)\circ P)$. It follows that
\[
\begin{aligned}
\|g_\la(u)-g_\la(v)\|_\la
 &\leq \alpha\|k_\la(u)-k_\la(v)-Dk_\infty(u_\infty)(P(u-v))\|_\la\\
 &\leq \alpha\|k_\la(u)-k_\la(v)-Dk_\la(u_\infty)(u-v)\|_\la\\
 &\hspace{1cm} +\alpha\|Dk_\la(u_\infty)(u-v)-Dk_\infty(u_\infty)(P(u-v))\|_\la\\
 &\overset{\eqref{satz2:4}}{\leq}
    \alpha\frac{1}{4\alpha}\|u-v\|_\la
    +\alpha\|Dk_\la(u_\infty)-Dk_\infty(u_\infty)\circ P\|_\la\cdot\|u-v\|_\la\\
 &\overset{\eqref{satz2:2}}{\leq}\frac14\|u-v\|_\la+\frac14\|u-v\|_\la
  = \frac12\|u-v\|_\la.
\end{aligned}
\]
We also have
\begin{equation*}
\|g_\la(u_\infty)-u_\infty\|_\la
 \leq \|L^{-1}\|_\la\cdot\|f_\la(u_\infty)\|_\la
 \overset{\eqref{satz2:3}}{\leq} \alpha\frac{\de}{2\alpha} = \frac{\de}{2}
\end{equation*}
hence, for $u\in B_{\de,\la}(u_\infty)$ there holds:
\[
\begin{aligned}
\|g_\la(u)-u_\infty\|_\la
 &\leq \|g_\la(u)-g_\la(u_\infty)\|_\la+\|g_\la(u_\infty)-u_\infty\|_\la\\
 &\leq \frac12\|u-u_\infty\|_\la+\frac{\de}{2}
  \leq \frac{\de}{2}+\frac{\de}{2} = \de.
\end{aligned}
\]
Therefore $g_\la$ maps $B_{\de,\la}(u_\infty)$ into itself.
\end{altproof}

\begin{altproof}{Theorem~\ref{thm:nondeg1}}
As in the proof of Theorem~\ref{thm:index1a} one sees that $J_\la$ is
of class $\cC^2$ and that $(J_5)$ holds. Therefore Theorem~\ref{thm:nondeg1} 
follows from Theorem~\ref{thm:nondeg2}. 
\end{altproof}

%
%

%

{\sc Address of the authors:}\\[1em]
 Thomas Bartsch, Mona Parnet\\
 Mathematisches Institut\\
 University of Giessen\\
 Arndtstr.\ 2\\
 35392 Giessen\\
 Germany\\
 Thomas.Bartsch@math.uni-giessen.de\\
 Mona.Parnet@math.uni-giessen.de

\end{document}